\documentclass[12pt]{amsart}
\usepackage[T1]{fontenc}
\usepackage[utf8]{inputenc}
\usepackage{fourier}
\usepackage[active]{srcltx}
\usepackage{a4wide}
\usepackage{amsthm,amsfonts,amsmath,mathrsfs,amssymb}
\usepackage{dsfont}
\def\be{\begin{equation}}
\def\ee{\end{equation}}

\newtheorem*{completeness*}{Completeness property}
\newtheorem*{theorem*}{Theorem}
\newtheorem{theorem}{Theorem}

\newtheorem*{proposition*}{Proposition}
\newtheorem{lemma}{Lemma}
\newtheorem{corollary}{Corollary}
\theoremstyle{remark}

\newcommand{\nc}{\newcommand}

\newcommand{\N}{{\mathbb N}}

\nc{\supp}{\operatorname{supp}}
\nc{\dif}{\operatorname{d}} \nc{\im}{\operatorname{i}}
\nc{\Hi}{{\mathscr{H}}^\infty} \nc{\Ht}{{\mathscr{H}}^2}
\nc{\Hone}{{\mathscr{H}}^1} \nc{\ol}{\overline} \nc{\bz}{\mathbf{z}}
\nc{\bw}{\mathbf{w}} \nc{\eps}{\varepsilon}

\allowdisplaybreaks[3]

\begin{document}
\title[Large GCD sums and extreme values of the Riemann zeta function]{Large GCD sums and extreme values of the Riemann zeta function}
\author{Andriy Bondarenko}
\address{Department of Mathematical Analysis\\ Taras Shevchenko National University of Kyiv\\
Volody- myrska 64\\ 01033 Kyiv\\ Ukraine}
\address{Department of Mathematical Sciences \\ Norwegian University of Science and Technology \\ NO-7491 Trondheim \\ Norway}

\email{andriybond@gmail.com}
\author[Kristian Seip]{Kristian Seip}
\address{Department of Mathematical Sciences \\ Norwegian University of Science and Technology \\ NO-7491 Trondheim \\ Norway}
\email{seip@math.ntnu.no}
\thanks{Research supported by Grant 227768 of the Research Council of Norway. }
\subjclass[2010]{11M06, 11C20}
\maketitle

\begin{abstract}
It is shown that  
the maximum of $|\zeta(1/2+it)|$ on the interval $T^{1/2}\le t \le T$ is at least 
$\exp\left((1/\sqrt{2}+o(1)) \sqrt{\log T \log\log\log T/\log\log T}\right)$. Our proof uses Soundararajan's resonance me- thod and a certain large GCD sum. The method of proof shows that the absolute constant $A$ in the inequality
\[
\sup_{1\le n_1<\cdots < n_N} \sum_{k,{\ell}=1}^N\frac{\gcd(n_k,n_{\ell})}{\sqrt{n_k
n_{\ell}}}  \ll N  \exp\left(A\sqrt{\frac{\log N \log\log\log N}{\log\log N}}\right), \]
established in a recent paper of ours, cannot be taken smaller than $1$. 
\end{abstract}

\section{Introduction}

This paper establishes the following new estimate for extreme values of the Riemann zeta function $\zeta(s)$ on the half-line.

\begin{theorem} \label{extreme}
Let $0<\beta<1$ be given and let $c$ be a positive number less than 
$\sqrt{\min(1/2, 1-\beta)}$. If $T$ is sufficiently large, then there exists a $t$, $T^{\beta} \le t \le T$, such that
\[ \left|\zeta\Big(\frac{1}{2}+it\Big)\right| \ge \exp\left(c\sqrt{\frac{\log T \log\log\log T}{\log\log T}}\right). \]
\end{theorem}
The best lower estimate for extreme values of \mbox{$|\zeta(1/2+it)|$} known previously was obtained in 2008 by Soundararajan \cite{S} who proved that
\[  \left|\zeta\Big(\frac{1}{2}+it\Big)\right| \ge \exp\left((1+o(1))\sqrt{\frac{\log T }{\log\log T}}\right)\]
holds for some $t$, $T\le t \le 2T$, if $T$ is large enough. In 1977, Montgomery \cite{M} had proved, assuming the Riemann Hypothesis, that 
there exist arbitrarily large $t$ such that \[ \left|\zeta\Big(\frac{1}{2}+it\Big)\right| \gg  \exp\left(c \sqrt{\frac{\log t }{\log\log t}}\right)\]
with $c=1/20$. This result was proved unconditionally at the same time by Balasubramanian and Ramachandra with a larger value of $c$ (see \cite{BR} and \cite{S}). Earlier milestones were bounds of the order of magnitude $\exp\Big(c(\varepsilon) (\log t)^{1/2-\varepsilon}\Big)$ and $\exp\Big(c\sqrt{\log t}/\log\log t\Big)$ obtained respectively by Titschmarsh in 1928 \cite{Ti} and Levinson in 1972 \cite{L}. We refer to Bourgain's recent paper \cite{B} for the best known upper bound $|\zeta(1/2+it)|\le t^{13/84+\varepsilon}$.

The proof of Theorem~\ref{extreme} will use the resonance method introduced by Soundararajan \cite{S}. The main new ingredient of the proof is a certain large greatest common divisor (GCD) sum related to our recent work \cite{BS}. In the latter paper,  we found that there exists an absolute constant $A$ less than 7 such that
\begin{equation}\label{gcdb}
\sum_{k,\ell=1}^N\frac{\gcd(n_k,n_{\ell})}{\sqrt{n_k
n_{\ell}}} \le N \exp\left(A\sqrt{\frac{\log N \log\log\log N}{\log\log N}}\right)
\end{equation}
for arbitrary integers $1\le n_1<\cdots <n_N$ and $N$ sufficiently large.
The following result is essential for the proof of Theorem~\ref{extreme} and will as well lead to the conclusion that \eqref{gcdb} is optimal in the sense that  it does not hold if $A<1$.
\begin{theorem}\label{gcdexample} Suppose that $0<\gamma<1$ is given. 
Then for every sufficiently large integer $N$ there exist integers $1\le n_1<\cdots < n_N$ and positive numbers $c_1,..., c_N$ such that
\begin{equation}\label{gcdc}
\sum_{k,{\ell}=1}^N c_k c_{\ell} \frac{\gcd(n_k,n_{\ell})}{\sqrt{n_k
n_{\ell}}} \ge \Big(\sum_{j=1}^N c_j^2\Big)\cdot \exp\left(\gamma\sqrt{\frac{\log N \log\log\log N}{\log\log N}}\right).
\end{equation}
\end{theorem} 

Theorem~\ref{gcdexample} belongs to the study of GCD sums of the form
\begin{equation}\label{gcda}
\sum_{k,\ell=1}^N\frac{(\gcd(n_k,n_{\ell}))^{2\sigma}}{(n_k
n_{\ell})^\sigma}
\end{equation}
and the associated matrices $\big(\frac{(\gcd(n_k,n_{\ell}))^{2\sigma}}{(n_k
n_{\ell})^\sigma}\big)$ for $\sigma>0$. This topic begins with G\'{a}l's theorem \cite{G} which asserts that $CN(\log \log N)^2$ is an optimal upper bound for \eqref{gcda} when $\sigma=1$, with $C$ an absolute constant independent of $N$ and the distinct positive integers $n_1,...,n_N$. In our case $\sigma=1/2$,
the first estimate was found by Dyer and Harman \cite{DH}, showing that the sum in \eqref{gcdb} is bounded by $\exp(C \log N/\log\log N)$.  The better bound $ \exp(C\sqrt{\log N \log\log N})$ was found by Aistleitner, Berkes, and Seip in  \cite{ABS} by a method that also led to a full solution when $1/2<\sigma <1$. For further details on the history of GCD sums of the form \eqref{gcda}, we refer to \cite{BS} or the recent paper \cite{LR} which gives the optimal constant $C$ in G\'{a}l's theorem and a simplified proof in the range $1/2<\sigma\le 1$.

The initial motivation for estimating GCD sums originated in Koksma's work in the metric theory of diophantine approximation \cite{K}. See Chapter~3 of Harman's book \cite{H} for a comprehensive treatment of applications in this area. In recent years, estimates of GCD sums have led to solutions to some longstanding problems regarding the almost everywhere convergence and boundedness of systems of dilated functions \cite{ABS, LR}. 
The explicit link to extreme values of $\zeta(s)$ came into light in Aistleitner's recent paper \cite{A}, which combined estimates for certain large GCD sums with Hilberdink's version of the resonance method \cite{H}. This led Aistleitner to a new proof of Montgomery's $\Omega$-results for $\zeta(\sigma+it)$ in the range $1/2<\sigma<1$ \cite{M}. 

Inspired by Aisleitner's work, the proof of Theorem~\ref{gcdexample} relies on our understanding of extremal sets of square-free numbers as described in \cite{BS}. We are not able to identify an actual extremal set, so that, strictly speaking, we are dealing only with ``nearly maximal'' GCD sums. But we come sufficiently close to reach the desired conclusion that our estimate from \cite{BS} is optimal, and we are thus led to  what appears to be an essentially optimal ``resonating'' Dirichlet polynomial in Soundararajan's method.

The version of the resonance method used in this paper can also be employed to prove existence of large values of $|\zeta(\sigma+it)|$ in the range $1/2 < \sigma < 1$. In particular, we may offer a new approach to Montgomery's result  \cite {M} which asserts that there exists a positive constant $c$ depending on $\sigma$ such that for all sufficiently large $T$, we have   
\begin{equation} \label{lowerM} \left|\zeta\big(\sigma+it\Big)\right| \ge \exp\left(c \frac{(\log T)^{1-\sigma} }{(\log\log T)^\sigma}\right)\end{equation}
for some $t$ in the range $T^\beta\le t \le T$ and some $\beta$, $0<\beta<1$. 
This is of some interest because the best exponent previously known for Montgomery's estimates, namely $c=c_0/(1-\sigma)$ for some $c_0=c_0(\sigma) < 0.17$ \cite{RS}, can be improved notably by our method. Indeed, we may replace this $c$ by a constant $\nu_\sigma$ enjoying the following asymptotic estimates:
\begin{equation}\label{montgomery} \nu_\sigma=\begin{cases} (1+o(1))(1-\sigma)^{-1}, & \sigma \nearrow 1, \\
(1/\sqrt{2}+o(1))\sqrt{|\log(2\sigma-1)|}, & \sigma\searrow 1/2. \end{cases}\end{equation}
The precise statement is that whenever $c<\nu_\sigma$, there exists a $\beta$, 
$1-\sigma\le  \beta < 1$ such that \eqref{lowerM} holds for some $t$, $T^\beta\le t \le T$.
This result gives a ``smooth'' transition between Montgomery's estimates and the respective endpoint cases represented by Levinson's classical bound for $\sigma=1$ \cite{GS, L} and our Theorem~\ref{extreme}; we also note that a precise conjecture from Lamzouri's paper \cite[Remark 2, p. 5454]{La} is consistent with our asymptotic estimate for $\nu_\sigma$ being sharp when $\sigma\nearrow 1$. The proof of \eqref{montgomery}, to be given in a forthcoming publication along with some additional applications of our version of the resonance method, shows that the smoothness of the numbers involved in an essentially optimal resonator decreases in an interesting and nontrivial way when $\sigma$ decreases from 1 to 1/2. 
 
 We will prove Theorem~\ref{gcdexample} in the next section, where we also show how this result leads to the failure of \eqref{gcdb} for $A<1$. Section~\ref{proofR} gives the proof of Theorem~\ref{extreme}, as well as a remark indicating that \eqref{gcdb} is a severe obstacle for further improvements using the resonance method. 
 
\section{Proof of Theorem~\ref{gcdexample}}\label{secttwo}

We begin by fixing a large integer $N$. To simplify the writing, we will use the short-hand notation $\log_2 N:=\log\log N$, $\log_3 N:=\log\log\log N$, and $\log_4 N:=\log\log\log \log N$.

We will construct $c_j$ and $n_j$ satisfying~\eqref{gcdc} using a suitable multiplicative function. To this end, we let $P$ be the set of all primes $p$ such that 
\[ e\log N\log_2 N< p \le \log N\exp( (\log_2 N)^{\gamma})\log_2 N. \]
We will choose $c_j=f(n_j)$ for suitable $n_j$, where $f(n)$ is the multiplicative function supported on the set of square-free numbers with
$$
f(p):=\sqrt{\frac{\log N \log_2N}{\log_3 N}}\frac{1}{\sqrt{p}(\log p-\log_2N-\log_3N)}
$$
for $p$ in $P$ and $f(p)=0$ otherwise. 
The idea to consider this function comes from our choice of weights $w_j$ in the proof of \cite[Lemma 4]{BS}. These weights played a crucial role in an application of the Cauchy--Schwarz inequality. 

We find that
\begin{align*}
\sum_{k, \ell=1}^Nf(n_k)f(n_{\ell})\frac{\gcd(n_k,n_{\ell})}{\sqrt{n_k
n_{\ell}}}& \ge\sum_{k=1}^N\frac{f(n_k)}{\sqrt{n_k}}\sum_{n_{\ell}|n_k}f(n_{\ell})\frac{\gcd(n_k,n_{\ell})}{\sqrt{n_{\ell}}} \\
& =\sum_{k=1}^N\frac{f(n_k)}{\sqrt{n_k}}\sum_{n_{\ell}|n_k}f(n_{\ell})\sqrt{n_{\ell}}. \end{align*}
This estimate leads us to consider the quantity
\begin{equation}\label{an}
A_N:=\frac{1}{\sum_{i\in\N}f(i)^2}\sum_{n\in\N}\frac{f(n)}{\sqrt{n}}\sum_{d|n}f(d)\sqrt{d}.
\end{equation}
Since $f$ is a multiplicative function, we obtain by induction that  
\begin{equation} \label{AN}
A_N=\prod_{p\in P}\frac{1+f(p)^2+f(p)p^{-1/2}}{1+f(p)^2}.
\end{equation}

We now claim that Theorem~\ref{gcdexample} is an immediate consequence of the following two lemmas.
\begin{lemma}
\label{lem1}
We have that
\[
A_N\ge \exp\left((\gamma+o(1))\sqrt{\frac{\log N \log_3 N}{\log_2 N}}\right)
\]
when $N\to \infty$.
\end{lemma}
In the second lemma, we use the following terminology: A set of positive integers $\mathcal{M}$ is said to be divisor closed if $d$ is in $\mathcal{M}$ whenever $m$ is in $\mathcal{M}$ and $d$ divides $m$.
\begin{lemma}
\label{lem2}
There exists a divisor closed set of integers $\mathcal{M}$ of cardinality at most $N$ such that
\begin{equation}
\label{i2}
\frac{1}{\sum_{i\in\N}f(i)^2}\sum_{n\in\N,n\not\in\mathcal{M}}\frac{f(n)}{\sqrt{n}}\sum_{d|n}f(d)\sqrt{d}=o(A_N),\quad N\to\infty.
\end{equation}
\end{lemma}
Indeed, since the set $\mathcal{M}$ in Lemma~\ref{lem2} is divisor closed, we infer from these two lemmas that
\[ \frac{1}{\sum_{i\in\N}f(i)^2}\sum_{m,n\in\mathcal{M}}f(m)f(n) \frac{\gcd(m,n)}{\sqrt{mn}} \ge (1+o(1))A_N,\quad N\to\infty, \]
and hence Theorem~\ref{gcdexample} has been established.
\begin{proof}[Proof of Lemma~\ref{lem1}]
Since $f(p)<(\log_3N)^{-1/2}$ for all $p$ in $P$, it follows from \eqref{AN} that
\[
A_N=\prod_{p\in P}\frac{1+f(p)^2+f(p)p^{-1/2}}{1+f(p)^2}=\exp\left((1+o(1))\sum_{p\in P}\frac{f(p)}{\sqrt{p}}\right).
\]
Now the conclusion of the lemma is obtained from the following computation:
\begin{align*}
\sum_{p\in P}\frac{f(p)}{\sqrt{p}} & =\sqrt{\frac{\log N \log_2N}{\log_3 N}}\sum_{p\in P}\frac{1}{p(\log p-\log_2N-\log_3N)} \\
& =(1+o(1))\sqrt{\frac{\log N \log_2N}{\log_3 N}}\int_{e\log N\log_2 N}^{\log N\exp( (\log_2 N)^{\gamma})\log_2 N}\frac{1}{x\log x(\log x-\log_2N-\log_3N)}dx \\
& = (1+o(1))\sqrt{\frac{\log N \log_2N}{\log_3 N}}\int_{1+\log_2N+\log_3N}^{\log_2N+(\log_2N)^{\gamma}+\log_3N}\frac{1}{t(t-\log_2N-\log_3N)}dt \\
& =(\gamma+o(1))\sqrt{\frac{\log N \log_3 N}{\log_2 N}}.
\end{align*}
\end{proof}
\begin{proof}[Proof of Lemma~\ref{lem2}]
To prove the lemma, we introduce some new notation.
Let $P_k$ be the set of all primes $p$ such that $e^k\log N\log_2N<p\le e^{k+1}\log N\log_2N$, $k=1,\ldots,[(\log_2 N)^{\gamma}]$.
Fix $1 < a < 1/\gamma$. Then let $M_k$ be the set of integers that have at least $\frac{a\log N}{k^2\log_3N}$ prime divisors in $P_k$,
and let $M'_k$ be the set of integers from $M_k$ that have prime divisors only in $P_k$.
Finally, set 
\[ \mathcal{M}:=\supp(f)\setminus\bigcup_{k=1}^{[(\log_2N)^{\gamma}]}M_k.\] In other words,
$\mathcal{M}$ is the set of square-free numbers $n$ that have at most $\frac{a\log N}{k^2\log_3N}$
divisors in each group $P_k$. It is is clear that $\mathcal{M}$ is divisor closed. 

We now estimate the cardinality of $\mathcal{M}$.
To this end, by the bounds
\[ \sqrt{2\pi} \nu^{\nu+1/2}e^{-\nu} \le \nu! \le e \nu^{\nu+1/2} e^{-\nu}, \] 
valid for all positive integers $\nu$, we see that for $m$ large enough we have
\[
\binom{m}{n}\le\frac{m^{m+1}}{n^n(m-n)^{m-n}}
\]
and hence
\begin{equation} \label{bin1} 
\binom{m}{n}\le m\left(1+\frac{n}{m-n}\right)^{m-n}\left(\frac{m}{n}\right)^n\le\exp\left(n(\log m-\log n)+n+\log m\right).
\end{equation}
We will also use the fact that 
\begin{equation} \label{bin2}
\frac{\binom{m}{n}}{\binom{m}{n-1}}=\frac{m-n+1}{n} \ge 2 
\end{equation}
when $m\ge 3n-1$.
By the prime number theorem, the cardinality of each $P_k$ is at most $e^{k+1}\log N$, and we therefore get, using first
\eqref{bin2} and then \eqref{bin1},
that
\begin{align*}
|\mathcal{M}| & \le\prod_{k=1}^{[(\log_2N)^{\gamma}]}\sum_{j=0}^{\big[\frac{a\log N}{k^2\log_3N}\big]} \binom{\big[e^{k+1}\log N\big]}{j} 
  \le \prod_{k=1}^{[(\log_2N)^{\gamma}]} 2 \binom{\big[e^{k+1}\log N\big]}{\big[\frac{a\log N}{k^2\log_3N}\big]} \\
& \le\exp\left(\sum_{k=1}^{[(\log_2N)^{\gamma}]}\left(1+\frac{a \log N}{k^2\log_3N}\big(k+2+\log_4N+2\log k\big)+k+1+\log_2N\right)\right)\le N
\end{align*}
when $N$ is large enough, since $a\gamma<1$. 

To prove~\eqref{i2}, we begin by noting that
\begin{equation}
\frac{1}{A_N\sum_{i\in\N}f(i)^2}\sum_{n\in\N,n\not\in\mathcal{M}}\frac{f(n)}{\sqrt{n}}\sum_{d|n}f(d)\sqrt{d}\le
\label{i2.5}
\frac{1}{A_N\sum_{i\in\N}f(i)^2}\sum_{k=1}^{[(\log_2N)^{\gamma}]}\sum_{n\in M_k}\frac{f(n)}{\sqrt{n}}\sum_{d|n}f(d)\sqrt{d}.
\end{equation}
Now for each $k=1,\ldots,[(\log_2N)^{\gamma}]$ we have that
\begin{align}
\frac{1}{A_N\sum_{i\in\N}f(i)^2}\sum_{n\in M_k}\frac{f(n)}{\sqrt{n}}\sum_{d|n}f(d)\sqrt{d}& =\frac{1}{\prod_{p\in P_k}(1+f(p)^2+f(p)p^{-1/2})}\sum_{n\in M'_k}\frac{f(n)}{\sqrt{n}}\sum_{d|n}f(d)\sqrt{d} \nonumber \\
\label{i3}
& \le\frac{1}{\prod_{p\in P_k}(1+f(p)^2)}\sum_{n\in M'_k}f(n)^2\prod_{p\in P_k}\left(1+\frac{1}{f(p)\sqrt{p}}\right).
\end{align}
To deal with the product to the right in \eqref{i3}, we make the following computation:
\begin{align}
\prod_{p\in P_k}\left(1+\frac{1}{f(p)\sqrt{p}}\right)& =\prod_{e^k\log N\log_2N<p\le e^{k+1}\log N\log_2N}\left(1+(\log p-\log_2N-\log_3N)\sqrt{\frac{\log_3N }{\log N\log_2 N}}\right) \nonumber
\\
\nonumber
& \le\left(1+(k+1)\sqrt{\frac{\log_3N }{\log N\log_2 N}}\right)^{e^{k+1}\log N}
\le \exp\left((k+1)e^{k+1} \sqrt{\frac{\log N \log_3N }{\log_2 N}} \right)\\
 & =\exp\left(o\left(\frac{\log N}{\log_3N}\right)\frac{1}{k^2}\right), \label{i4}
\end{align}
where the latter relation holds simply because $k\le (\log_2 N)^{\gamma}$. Since every number in $M_k'$ has at least
$\frac{a\log N}{k^2 \log_3 N}$ prime divisors and $f(n)$ is a multiplicative function, it therefore follows that
$$
\sum_{n\in M'_k}f(n)^2\le   b^{-a\frac{\log N}{k^2\log_3N}}\prod_{p\in P_k}(1+bf(p)^2)
$$
whenever $b> 1$ and hence
\begin{equation}
\label{i5}
\frac{1}{\prod_{p\in P_k}(1+f(p)^2)}\sum_{n\in M'_k}f(n)^2\le b^{-a\frac{\log N}{k^2\log_3N}}\exp\left(\sum_{p\in P_k}(b-1) f(p)^2\right).
\end{equation}
Finally,
\begin{align*}
\sum_{p\in P_k} f(p)^2& =\frac{\log N \log_2N}{\log_3 N}\sum_{p\in P_k}\frac{1}{p(\log p-\log_2N-\log_3N)^2}
\\
& \le (1+o(1))\frac{\log N \log_2N}{\log_3 N}\int_{e^k\log N\log_2 N}^{e^{k+1}\log N\log_2 N}\frac{1}{k^2x\log x}dx \\
& \le (1+o(1))\frac{\log N}{k^2\log_3N}.
\end{align*}
Combining the last inequality with~\eqref{i5} and~\eqref{i4}, we get that~\eqref{i3} is at most 
$$
\exp\left((b-1-a \log b +o(1))\frac{\log N}{k^2\log_3N}\right).
$$
Choosing $b$ sufficiently close to 1, we obtain $b-1-a\log b <0$. Returning to~\eqref{i2.5}, we therefore see that the desired relation~\eqref{i2} has been established.
\end{proof}

We close this section by showing that Theorem~\ref{gcdexample} implies that \eqref{gcdb} fails for $A<1$.
\begin{corollary}
For every $\gamma$, $0<\gamma<1$ and any given positive number $R$, there exists an integer $N\ge R$ and associated integers $1\le n_1<\cdots < n_N$ such that 
\begin{equation}\label{gcdl}
\sum_{k,\ell=1}^N\frac{\gcd(n_k,n_{\ell})}{\sqrt{n_k
n_{\ell}}} > N \exp\left(\gamma\sqrt{\frac{\log N \log\log\log N}{\log\log N}}\right).
\end{equation}
\end{corollary}

\begin{proof} We introduce the two quantities 
\[ \Gamma(N):= \sup_{1\le n_1<\cdots < n_N} N^{-1} \sum_{k,{\ell}=1}^N\frac{\gcd(n_k,n_{\ell})}{\sqrt{n_k n_{\ell}}} \]
and
\[ \Lambda(N):= \sup_{1\le n_1<\cdots < n_N} \sup_{(c_1,..., c_N)\neq 0}
\frac{\sum_{k,{\ell}=1}^Nc_kc_{\ell}\frac{\gcd(n_k,n_{\ell})}{\sqrt{n_k
n_{\ell}}}}{\sum_{j=1}^N c_j^2}, \]
which are related by the two inequalities
\begin{equation}\label{two} \Gamma(N)\le \Lambda(N)\le (e^2+1)(\log N+2) \max_{n\le N} \Gamma(n). \end{equation}
Here the left inequality is obvious and the right inequality was established in \cite[Theorem 5]{ABS}.
It now follows from Theorem~\ref{gcdexample}, the right inequality in \eqref{two}, and \eqref{gcdb} that for every $\gamma$, $0<\gamma < 1$, there exists an absolute constant $\delta$, $0<\delta < 1$, such that for every sufficiently large $N$ there is an $n$ in $[N^\delta, N]$ for which
\[ \exp\left(\gamma\sqrt{\frac{\log N \log_3 N}{\log_2 N}}\right) \le \Gamma(n). \]
This shows that whenever $0<\gamma<1$, there must exist arbitrarily large $N$ and associated integers $1\le n_1< \cdots < n_N$ such that \eqref{gcdl} holds.
\end{proof}

Note that if we knew that $N\mapsto \Gamma(N)$ is an increasing function, then we could immediately have made the stronger conclusion that
$ \exp\Big(\gamma \sqrt{\log N \log_3 N/\log_2 N}\Big) \le \Gamma(N)$ for all sufficiently large $N$. This inequality does indeed hold, but its proof requires additional technicalities which we choose not to supply since they are not needed for the proof of Theorem~\ref{extreme}.  

\section{Proof of Theorem~\ref{extreme}}\label{proofR}

Following Soundararajan's method, we seek a Dirichlet polynomial
\[ R(t)=\sum_{m\in \mathcal{M}'} r(m) m^{-it} \]
with $|\mathcal{M}'| \le T^{\kappa}$ for some $\kappa\le 1/2$ and positive coefficients $r(n)$ that ``resonates''
with $\zeta(1/2+i t)$. As smoothing function we will use the Gaussian $\Phi(t):=e^{-t^2/2}$, and we define
\begin{align*} M_1(R,T)& :=\int_{T^{\beta}\le |t|\le T} |R(t)|^2 \Phi\Big(\frac{\log T}{T} t\Big) dt, \\
M_2(R,T)& :=\int_{T^{\beta}\le |t| \le T}  \zeta(1/2+it) |R(t)|^2 \Phi\Big(\frac{\log T}{T} t\Big) dt.
\end{align*}
Then 
\begin{equation} \label{plain} \max_{T^{\beta}\le t \le T} \big|\zeta(1/2+it)\big| \ge \frac{|M_2(R,T)|}{M_1(R,T)}, \end{equation}
and the goal is therefore to maximize the ratio on the right-hand side of \eqref{plain}. Our particular choice of smoothing function is of course not important; the properties that we will need, are that the Fourier transform $\widehat{\Phi}$ of $\Phi$ is positive and that both $\Phi$ and $\widehat{\Phi}$ decay fast.

Before going further, we would like to explain the main new idea of Aistleitner's proof \cite{A} which we will adapt to Soundararajan's method. The purpose of the resonator is to pick out terms in the approximating sum $\sum_{n\le T} n^{-1/2-it}$ that contribute substantially to the average size of $\zeta(1/2+it)$. Since we integrate $\zeta(1/2+it)$ against $|R(t)|^2$, the frequencies that pick out large contributions are of the form $\log (m/n)$ with $m, n$ in $\mathcal{M}'$. This means that what really matters is the size of the ratios $m/n$ rather than that of the integers $m$ and $n$ themselves. We have therefore, following Aistleitner, abandoned the restriction from Soundararajan's approach that the resonator $R$ have the same length as the approximating sum $\sum_{n\le T}n^{-1/2-it}$. Our proof will reveal the interesting point that the terms in the latter sum corresponding to the integers from $\mathcal{M}\cap [1,T]$ give a larger contribution than those picked by Soundararajan's resonator. 
There are however nontrivial technical difficulties associated with a resonator that is much longer than the approximating sum $\sum_{n\le T} n^{-1/2-it}$, as discussed in detail in \cite{A}. We will now explain how such a Dirichlet polynomial $R(t)$ can be chosen to overcome these obstacles.   

We retain the notation from the preceding section with $N=[T^{\kappa}]$, choosing $0<\gamma<1$ such that $c<\gamma \sqrt{\min(1/2, 1-\beta)}$. Following an idea from \cite{A}, we let $\mathcal{J}$ be the set of integers $j$ such that
\[ \Big[(1+T^{-1})^j,(1+T^{-1})^{j+1}\Big)\bigcap \mathcal{M} \neq \emptyset,  \]
and let $m_j$ be the minimum of  $\big[(1+T^{-1})^j,(1+T^{-1})^{j+1}\big)\bigcap \mathcal{M}$ for $j$ in $\mathcal{J}$. Then set
\[ \mathcal{M}':= \big \{ m_j: \ j\in \mathcal{J} \big\}\]
and 
\[ r(m):= \left(\sum_{n\in \mathcal{M}, 1-T^{-1}(\log T)^2 \le n/m \le 1+T^{-1}(\log T)^2} f(n)^2\right)^{1/2} \] 
for every $m$ in $\mathcal{M}'$. Note that plainly $|\mathcal{M}'|\le |\mathcal{M}| \le N$. 

The proof of Theorem~\ref{extreme} will also require two additional estimates which we state as separate lemmas. The first is related to the quantity $A_N$ defined in \eqref{an}.

\begin{lemma}
\label{lem3}
Let $\mathcal{M}$ be the set constructed in the proof of Lemma~\ref{lem2} and $\varepsilon$ be a positive number. Then
\begin{equation}
\label{iii}
\frac{1}{\sum_{i\in\N}f(i)^2}\sum_{n\in\mathcal{M}}\frac{f(n)}{\sqrt{n}}\sum_{d|n,\,d\le n/N^{\varepsilon}}f(d)\sqrt{d}=o(A_N),\quad N\to\infty,
\end{equation}
where the implicit constant only depends on $\varepsilon$.
\end{lemma}

\begin{proof}
We have 
\[
\sum_{n\in\mathcal{M}}\frac{f(n)}{\sqrt{n}}\sum_{d|n,\,d\le n/N^\varepsilon}f(d)\sqrt{d}=
\sum_{n\in\mathcal{M}}f(n)^2\sum_{k|n,\,k\ge N^\varepsilon}\frac{1}{f(k)\sqrt{k}}.
\]
It is therefore enough to show that for each $n$ in $\mathcal{M}$ we have
\[
\sum_{k|n,\,k>N^{\varepsilon}}\frac{1}{f(k)\sqrt{k}}=o(1),\quad N\to\infty.
\]
Finally, we obtain
\[
\sum_{k|n,\,k>N^{\varepsilon}}\frac{1}{f(k)\sqrt{k}}
\le N^{-\varepsilon/4}\sum_{k|n}\frac{1}{f(k)k^{1/4}} = N^{-\varepsilon/4}\prod_{ p|n}\left(1+\frac{1}{p^{1/4} f(p)}\right)=o(1)
\]
as required. The last identity is clear because $1/(p^{1/4} f(p))=o(1)$ uniformly for all $p$ in $P$ and the integer $n$ has at most  $2a\log N/\log_3 N$ prime divisors.
\end{proof}

Our second estimate deals with the action of the resonator in the interval $|t|\le T^{\beta}$.

\begin{lemma}
\label{lem4}
For an arbitrary positive number $M$, we have
\[
\Big| \sum_{1\le n\le T}n^{-1/2}\int_{-T^\beta}^{T^{\beta}}\left(\frac{M}{n}\right)^{it}\Phi\Big(\frac{\log T}{T} t\Big) dt\Big|
\ll\max(T^{\beta},T^{1/2}\log T),
\]
where the implicit constant is independent of $M$.
\end{lemma}
\begin{proof}
We begin by noting that 
\begin{equation}\label{alter} \big |\int_{-T^\beta}^{T^{\beta}}e^{i\lambda t}\Phi\Big(\frac{\log T}{T} t\Big)dt\big |\ll\min(T^{\beta},\frac{1}{|\lambda|}). \end{equation}
This follows from \cite[Lemma 4.3]{T}
if we set $F(t):=\lambda t$ and $G(t):=\Phi\Big(\frac{\log T}{T} t\Big)$ on the intervals $[-T^{\beta},0]$ and $[0,T^\beta]$.
We infer from \eqref{alter} that
\begin{equation} \label{starting}
\Big| \sum_{1\le n\le T}n^{-1/2}\int_{-T^\beta}^{T^{\beta}}\left(\frac{M}{n}\right)^{it}\Phi\Big(\frac{\log T}{T} t\Big) dt\Big| \ll\sum_{1\le n\le T}n^{-1/2}\min\left(T^{\beta},\frac{1}{\big|\log\frac{M}{ n}\big|}\right).
\end{equation}
For $M-M^{1/2}\le n\le M+M^{1/2}$, we use the bound $T^\beta$ for the minimum to the right, and for $n$ outside the interval $[\frac{M}{2},\frac{3M}{2}]$, we use that this minimum is $\ll 1$. Setting
\[
S_m:=[M-(m+1)M^{1/2},M-mM^{1/2}]\cup[M+mM^{1/2},M+(m+1)M^{1/2}],
\]
we therefore find that 
\[
\sum_{1\le n\le T}n^{-1/2}\min\left(T^{\beta},\frac{1}{\big|\log\frac{M}{ n}\big|}\right) 
\ll T^{\beta}+T^{1/2}+\sum_{1\le m \le \frac{M^{1/2}}{2}}\sum_{n\in S_m, n\le T}n^{-1/2}\frac{1}{\big|\log\frac{M}{ n}\big|}. \]
We may clearly assume that $M\le 2T$ since otherwise the latter sum is $0$. Using that
\[ 
\frac{1}{|\log\frac{M}{n}|}\ll \frac{M^{1/2}}{m}
\]
for $n$ in $S_m$, we therefore see that 
\[ \sum_{1\le m \le \frac{M^{1/2}}{2}}\sum_{n\in S_m}n^{-1/2}\frac{1}{\big|\log\frac{M}{ n}\big|} \ll \sum_{1\le m\le \frac{M^{1/2}}{2}} \frac{M^{1/2}}{m}\ll M^{1/2} \log M \le T^{1/2} \log T. \] Returning to \eqref{starting}, we obtain the assertion of the lemma.
\end{proof}

With the resonator and Lemma~\ref{lem3} and Lemma~\ref{lem4} in place, we now turn to the estimates for $M_1(R,T)$ and $M_2(R,T)$.
\begin{proof}[Proof of Theorem~\ref{extreme}] 
We begin by finding an upper bound for $M_1(R,T)$:
\begin{equation}\label{m1} M_1(R,T)\le \int_{-\infty}^\infty |R(t)|^2 \Phi\Big(\frac{\log T}{T} t\Big) dt=\frac{\sqrt{2\pi} T}{\log T} \sum_{m,n\in \mathcal{M}'} r(m)r(n) 
\Phi\Big(\frac{T}{\log T} \log \frac{m}{n} \Big) \end{equation}
since
\[ \widehat{\Phi}(x):=\int_{-\infty}^\infty \Phi(t)e^{-itx} dt= \sqrt{2\pi} \Phi(x).  \]
We find first that
\begin{align*} \sum_{m,n\in \mathcal{M'}, |m/n-1| > T^{-1} (\log T)^2} r(m) r(n) \Phi\Big(\frac{T}{\log T} \log \frac{m}{n} \Big)
& \ll N \Phi(\log T) \sum_{m\in \mathcal{M}'} r(m)^2 \\
& \ll N \Phi(\log T) (\log T)^2 \sum_{n\in \mathcal{M}} f(n)^2 \\
& =o(1) \sum_{n\in \mathcal{M}} f(n)^2, \end{align*}
where we used the Cauchy--Schwarz inequality, the definition of $r(m)$, and finally the rapid decay of $\Phi(t)$. 
Since 
\begin{align*} \sum_{m,n\in \mathcal{M}', |m/n-1|\le T^{-1}(\log T)^2}r(m) r(n)  
 & \ll  \sum_{j\in \mathcal{J}} r(m_j) \sum_{|j-k|\le (\log T)^2 } r(m_k) \\
&  \ll (\log T)^4 \sum_{n\in \mathcal{M}} f(n)^2, \end{align*}
we may therefore return to \eqref{m1} and conclude that
\begin{equation} \label{m1f} M_1(R,T) \ll T (\log T)^3 \sum_{n\in \mathcal{M}} f(n)^2. \end{equation} 

We next turn to the lower bound for $M_2(R,T)$. We use the classical approximation 
\begin{equation}\label{approx} 
 \zeta(1/2+i t)= \sum_{n\le T} n^{-1/2-it} - \frac{T^{1/2-it}}{1/2-it}+O(T^{-1/2}), \end{equation}
which is valid for $|t|\le T$ (see \cite[Theorem 4.11]{T}). Hence, using also the trivial estimate $|R(t)|^2 \le N \sum_{m\in \mathcal{M}'}r(m)^2$, we find that
\[ M_2(R,T)= \int_{T^{\beta}\le |t| \le T} \sum_{n\le T}n^{-1/2-it} |R(t)|^2 \Phi\Big(\frac{\log T}{T} t\Big) dt + O(T^{1/2+\kappa})(\log T)^3 \sum_{n\in \mathcal{M}}f(n)^2.\]
By Lemma~\ref{lem4}, we have
\begin{align*} \left|\int_{-T^\beta}^{T^{\beta}} \sum_{n\le T}n^{-1/2-it}  |R(t)|^2\Phi\Big(\frac{\log T}{T}t\Big) dt\right|&
\ll \max(T^{\beta},T^{1/2}\log T)  \sum_{m, n  \in \mathcal{M'}}r(m)r(n) \\  & \ll 
\max(T^{\beta},T^{1/2}\log T) T^\kappa (\log T)^3 \sum_{n\in \mathcal{M}}f(n)^2,\end{align*}
where we in the last step used the Cauchy--Schwarz inequality and the definition of $r(n)$.
We see that the right-hand side is $O(T(\log T)^4) \sum_{n\in \mathcal{M}}f(n)^2$ if we choose $\kappa=\min(1/2, 1-\beta)$.
Since trivially
\[  \int_{ |t| \ge T} \big|\sum_{n\le T}n^{-1/2-it}| \cdot |R(t)|^2 \Phi\Big(\frac{\log T}{T} t\Big) dt \ll 
o(1)  \sum_{n\in \mathcal{M}}f(n)^2\]
by the rapid decay of $\Phi(t)$, we therefore have
\begin{equation}\label{m2}  M_2(R,T)= I(R,T) + O(T (\log T)^4)  \sum_{n\in \mathcal{M}}f(n)^2, \end{equation}
where
\[ I(R,T):=\int_{-\infty}^{\infty}  \sum_{n\le T}n^{-1/2-it} |R(t)|^2 \Phi\Big(\frac{\log T}{T} t\Big) dt.\]
Computing as in the preceding case, we see that
\begin{align*}  I(R,T) & =\frac{\sqrt{2\pi} T}{\log T} \sum_{m,n\in \mathcal{M}'} \sum_{k\le T} \frac{r(m)r(n)}{\sqrt{k}} \Phi\Big(\frac{T}{\log T} \log \frac{km}{n} \Big) \\
& \ge \frac{\sqrt{2\pi} T}{\log T} \sum_{m,n\in \mathcal{M}'} \sum_{k\in \mathcal{M}, k\le T} \frac{r(m)r(n)}{\sqrt{k}} \Phi\Big(\frac{T}{\log T} \log \frac{km}{n}\Big). \end{align*}
Here we used that all the terms in the sum are positive, so that we could sum over a suitable subcollection of them. We will now do this a second time in such a way that we are able to relate the sum to the quantity 
\[ A_N=\frac{1}{\sum_{i\in\N}f(i)^2}\sum_{n\in\N}\frac{f(n)}{\sqrt{n}}\sum_{d|n}f(d)\sqrt{d}\] 
which was defined in \eqref{an}. To this end, we apply the following rule: For a given $k$ in $\mathcal{M}$, we consider all pairs $m',n'$ in $\mathcal{M}'$ such that $ |km'/n' -1|\le 3/T$. It follows from the Cauchy--Schwarz inequality that
\[ \sum_{m,n\in \mathcal{M}, mk=n, 0 \le m/m'-1 \le1/T,  0\le n/n'-1 \le 1/T}f(m)f(n)\le r(m')r(n') \]
and hence, by the definition of $\mathcal{M}'$, that
\[ \sum_{m,n\in \mathcal{M}, mk=n}f(m)f(n)\le \sum_{m',n'\in \mathcal{M}', |km'/n' -1|\le 3/T} r(m')r(n'). \]
Now dividing this inequality by $\sqrt{k}$ and summing over all $k$ in $\mathcal{M}\cap [1,T]$, we therefore get that
\[ \label{last} I(R,T) \gg \frac{T}{\log T} 
\sum_{n\in\mathcal{M}}\frac{f(n)}{\sqrt{n}}\sum_{d|n, d\ge n/T}f(d)\sqrt{d}. \]
In view of Lemma~\ref{lem2} and Lemma~\ref{lem3}, we see that we obtain
\[ I(R,T) \gg \frac{T}{\log T} A_N \sum_{n\in \mathcal{M}} f(n)^2;\]
here the application of Lemma~\ref{lem3} was justified because $N=[T^{\kappa}]$ which implies that we can choose $\varepsilon=2$ .
Using also Lemma~\ref{lem1} and returning to \eqref{m2}, we finally get
\begin{equation}\label{m2f} |M_2(R,T)|\gg \frac{T}{\log T}  \exp\left(\big(\gamma+o(1)\big)\sqrt{\kappa \frac{\log T \log_3 T}{\log_2 T}}\right) \sum_{n\in \mathcal{M}}f(n)^2. \end{equation}

We finish the proof by plugging \eqref{m1f} and \eqref{m2f} into \eqref{plain}. We then obtain
\[ \max_{T^{\beta}\le t \le T} \big|\zeta(1/2+it)\big| \ge  \exp\left(c \sqrt{\frac{\log T \log_3T}{\log_2 T}}\right) \] 
for all sufficiently large $T$ since we have $c<\sqrt{\kappa}$ by our choice of  $\kappa$.  \end{proof}

Note that it was used several times in the proof that the set $\mathcal{M}'$ has cardinality not exceeding $T^{\kappa}$ with $\kappa\le 1/2$. In Soundararajan's original proof, one had the same bound with $\kappa<1$ for this cardinality. In either case, it seems hard to dispense with a restriction like this. Taking into account the step leading to the crucial estimate \eqref{m2f}, we see that \eqref{gcdb} combined with the right inequality in \eqref{two} indicates that the only improvement we could hope for following Soundararajan's method is to find a slightly larger $c$ in Theorem~\ref{extreme}. 

\section*{Acknowledgements}

We are grateful to Christoph Aistleitner, Maksym Radziwi{\l\l}, and Oleksandr Rudenko for several fruitful discussions. We would also like to thank 
Maksym Radziwi{\l\l} for some pertinent comments on a preliminary version of the paper.

\end{document}